\documentclass{amsart}

\usepackage[all] {xy}
\usepackage[dvips]{graphicx}
\usepackage{graphics,color}
\usepackage{tikz}

\newtheorem{theorem}{Theorem}[section]
\newtheorem{corollary}[theorem]{Corollary}

\newtheorem{lemma}[theorem]{Lemma}
\newtheorem{proposition}[theorem]{Proposition}

\theoremstyle{definition}
\newtheorem{definition}[theorem]{Definition}

\newtheorem*{question}{Question}

\newcommand{\Mundo}{\mathfrak{X}^{1}(M)}

\newcommand{\N}{\mathbb{N}}
\newcommand{\R}{\mathbb{R}}
\newcommand{\Q}{\mathbb{Q}}
\newcommand{\eps}{\varepsilon}

\newcommand{\SG}{{\mathcal G}}

\newcommand{\SR}{{\mathcal R}}

\newcommand{\SU}{{\mathcal U}}
\newcommand{\SV}{{\mathcal V}}

\newcommand{\ack}{{\bf Acknowledgements. }}

\begin{document}

\title[The Specification Property for Flows]
      {The Specification Property for Flows from the Robust and Generic Viewpoint}

\author[A. Arbieto]{A. Arbieto}
\address{Instituto de Matem\'atica, Universidade Federal do Rio de Janeiro, P. O. Box 68530, 21945-970 Rio de Janeiro, Brazil.}
\email{arbieto@im.ufrj.br}
\thanks{The first author was partially supported by CNPq, FAPERJ (``Jovem Cientista do Nosso Estado'') and PRONEX/DS from Brazil.}

\author[L. Senos ]{L. Senos}
\address{Instituto de Matem\'atica, Universidade Federal do Rio de Janeiro, P. O. Box 68530, 21945-970 Rio de Janeiro, Brazil.}
\email{laurasenosl@hotmail.com}
\thanks{The second author was partially supported by CNPq, CAPES/PRODOC Grant from Brazil.}

\author[T. Sodero ]{T. Sodero}
\address{Instituto de Matem\'atica, Universidade Federal do Rio de Janeiro, P. O. Box 68530, 21945-970 Rio de Janeiro, Brazil.}
\email{tati\_sodero@im.ufrj.br}
\thanks{The third author was partially supported by CNPq.}

\subjclass{Primary: 37C10, 37D20; Secondary: 37C20}


\keywords{Specification Property; Anosov Flows}

\begin{abstract}
We  prove that if $X|_\Lambda$ has the weak specification property robustly, where $\Lambda$  is an isolated set, then $\Lambda$ is a  hyperbolic topologically mixing set and,  as a consequence, if $X$ is a vector field that has the weak specification property robustly on a closed manifold $M$, then the flow $X_t$ is a topologically mixing Anosov flow. Also we prove that there exists a residual subset $\SR \in \Mundo$ so that if $X \in \SR$ and $X$ has the weak specification property, then $X_t$ is an Anosov flow.
\end{abstract}

\maketitle

\section{Introduction}


The theory of dynamical systems is  motivated by the search of knowledge of the behavior of most of the orbits of a given dynamical system. Since the work of Smale \cite{Smale}, this question was answered in a satisfactory way for hyperbolic systems. Some of these systems are called Anosov, when the whole manifold possesses a hyperbolic structure. Thus, to know what dynamical properties lead to the presence of Anosov systems is an important issue in the theory dynamical systems.

In this spirit, we can ask the consequences of the existence of dynamical properties in a robust way. More specifically, as we say above, if they imply some differential properties of the system. There are many works on this subject, for instance in the work of Ma\~n\'e \cite{Ma}, where he proves that in dimension 2, if a diffeomorphism is transitive in a robust way then it is Anosov. Usually this results, and hence the notion of robustness, holds in the $C^1$-topology. There is an analogous result in the continuous time setting by Doering \cite{D}, where he proves that a $C^1$-robustly transitive vector field on a closed 3-manifold is Anosov. We note that, \emph{a fortiori}  this rules out the existence of singularities. In the semi-local case, there is a result by Morales-Pac{\'i}fico-Pujals \cite{MPP}, where they prove that an isolated compact invariant set of a vector field on a 3-manifold which is transitive in a robust manner is a \emph{sectional-hyperbolic} set. We refer the reader to the next section for the precise definitions. However, we remark that one of the most famous sectional-hyperbolic sets is the Lorenz attractor. This attractor is not hyperbolic due to the presence of singularities, but is still robust.

Another point of view is to look what consequences some dynamical properties have for \emph{most of} the dynamical systems. Of course, we must specify what notion of largeness we are adopting. Using the fact that the $C^1$ topology make the set of dynamical systems (diffeomorphisms or vector fields) a Baire space, we could investigate properties or consequences of the dynamics for a \emph{residual} subset of dynamical systems. By definition, a residual subset is a countable intersection of open and dense sets. The Baire property says that any residual subset is dense, thus large in a topological sense. However, by definition a finite intersection of residual subsets is also a residual subset. Thus, proving more and more generic properties (i.e. properties that holds in a residual subset) you can use them to prove new properties for generic systems. The generic theory of dynamical systems deals with this type of problem and will be exploited in this article. Of course we could investigate generic system in the $C^r$-topology, since it is also Baire. However, most of the perturbations tools are still only available in the $C^1$-topology, so we restrict our studies to this topology.

The notion of the specification property is due to Bowen in \cite{Bowen} and this has turned out to be a very important notion in the study of ergodic theory of dynamical systems on a compact metric space and on statistical mechanics. More specifically, Bowen shows in \cite{B2} that for an expansive homeomorphism satisfying specification property on a compact metric space we have unique equilibrium states. In \cite{Fr} Franco shows the analogous theorem for the case of a continuous flow. Haydn and Ruelle in \cite{HR} studied the consequences of expansiveness and specification property on statistical mechanics.

Morally, a diffeomorphism $f$ or a flow $X_t$ on a compact manifold $M$ satisfies the specification property if one can shadow distinct $n$ pieces of orbits, which are sufficiently time-spaced, by a single orbit. We say that the specification property is weak if $n = 2$. The precise definition of the weak specification property will be given in the next section. In fact, it is quite technical and seems to be very strong, but it is satisfied by many examples. Indeed, every topologically mixing compact locally maximal hyperbolic set for a smooth flow satisfies this property.

The weak specification property was investigated from the viewpoint of geometric theory of discrete dynamical systems by Sakai-Sumi-Yamamoto in \cite{Sakai}, where they characterize diffeomorphisms satisfying the weak specification property robustly as Anosov diffeomorphisms.

In this paper, we extend the results in \cite{Sakai} for vector fields. More precisely, we characterize flows satisfying the weak specification property robustly. Indeed, we prove that if a flow satisfies the weak specification property robustly then the flow is Anosov. Actually, this result follows from a semi-local result which says that if an isolated invariant compact set satisfy the weak specification property robustly then this set is a  hyperbolic set (see the next section).   We want to stress that some arguments used in the robust case follows the lines of \cite{Sakai}, specifically to show the hyperbolicity of periodic orbits. Even so, we also prove the hyperbolicity of singularities, performing a similar argument. This enable us to show that the set is sectional-hyperbolic. However, as we remark before, this is not sufficient to conclude hyperbolicity. Still, we use the weak specification property combined with Kupka-Smale's theorem \cite{Palis} and a version of Hayashi's connecting lemma \cite{H}, given in \cite{Gan-Wen}, to rule out singularities, except in the trivial case when the set reduces to a unique singularity, and with this we obtain hyperbolicity. We also show that the presence of the weak specification property for a generic vector field implies hyperbolicity, thus complementing the result in the robust case, once again we need to rule out singularities, we do that arguing as in the robust case but now using Pugh's general density theorem \cite{Pugh}.

We would like to point out that our results deals with manifolds with dimension bigger than two. In the 2-dimensional case, by a result of Peixoto \cite{Peixoto}, we know that the Morse-Smale flows form an open and dense subset of the set
of $C^1$-flows. As Morse-Smale flows can not be topologically mixing, we conclude that flows satisfying the specification property robustly may exist only on manifolds with dimension higher than two.


Following the proof of Proposition 23.20 in \cite{DGS},  any  continuous flow which   has the shadowing property and is topologically mixing, also has the weak specification property. Actually, expansiveness is an hypothesis of that proposition, but here we do not need it since we  do not require  that the shadow is a periodic orbit, this is the only place where expansiveness is used.

It was proved in \cite{KOMURO}  that a geometric Lorenz attractor has the shadowing property if and only if the first return map $f$ satisfies that $f(0) = 0$ and $f(1) = 1$. 
This has the following consequence, if there exists a topologically mixing geometrical Lorenz flow such that its first return map satisfies $f(0) = 0$ and $f(1) = 1$ then this flow has the weak specification property.

Our results says that this does not happen for generic topologically mixing geometrical Lorenz flows. However,  since the shadowing property is not implied by the weak-specification property, still could exists some non-generic Lorenz flows (or attractors) with this property. This motivates the following question:


\begin{question}
What are the geometrical Lorenz attractors that satisfy the weak specification property?
\end{question}

This paper is organized as follows in section 2 we give precise definition and enunciate the main results. In section 3, we collect some consequences of the weak specification property. In section 4, we analyze the hyperbolicity of the periodic orbits. In section 5, we show some generic properties and prove some consequences of the weak specification property in this context. In section 6, we deal with singularities. In section 7, we give the proof of the results which use robustness. In section 8, we give the proof of the generic result.   Finally, in section 9 we make some comments about our main results in other contexts.

\section{Statement of the Results}

Let $M^n$, $n \geq 3$, be a Riemannian closed manifold, i.e. compact and boundaryless, and $X$ be a vector field on $M$. We denote by $d$ the induced metric on $M$, we also define the set $B(x,\epsilon) = \{y \in M; d(x,y)< \epsilon \}$. We denote by $X_t$ the generated flow and by $X_{[a,b]}(x)$ the piece of orbit defined by the set $\{y \in M; X_t(x)=y, t \in [a,b]\}$. We say that $p$ is a periodic point, or it belongs to a periodic orbit, if there exists $T>0$ such that $X_T(p)=p$,  the period of $p$ is the first positive $T$ which satisfies this equation and  the set of periodic points will be denoted by $PO(X)$. We say that $\sigma$ is a singularity if $X(\sigma)=0$, the set of singularities is denoted by $Sing(X)$. The set of critical orbits of $X$ is $Crit(X)=PO(X)\cup Sing(X)$. As usual, we also denote by $O(p)$ the orbit of $p$. The omega limit set of $x$, denoted by $\omega(x)$, is the set of points $y \in M$ such that there exists a sequence $t_n \to \infty$ with $\lim_{t_n \to \infty}X_{t_n}(x) = y$. Similarly, the alpha limit set of $x$, denoted by $\alpha(x)$, is the set of points $y \in M$ such that there exists a sequence $t_n \to -\infty$ with $\lim_{t_n \to -\infty}X_{t_n}(x) = y$.

We say that $\sigma\in Sing(X)$ is hyperbolic if all of the eigenvalues of $DX(\sigma)$ has non-zero real part. A periodic point is hyperbolic if the eigenvalues of its Poincar\'e map do not belong to the unit circle. A vector field $X$ is said to be Kupka-Smale if any critical orbit  is hyperbolic and $W^s(\sigma_1)$ is transverse to $W^u(\sigma_2)$ where $\sigma_i$ are critical orbits of $X$. The set of the $C^1$ vector fields of $M$ is denoted by $\mathfrak{X}^1(M)$ and it is endowed with the $C^1$-topology.

Let $\Lambda$ be an invariant compact set of $M$. A \emph{specification} $S= (\tau,P)$ consists of a finite colection $\tau=\{I_1, \cdots , I_m\}$ of bounded intervals $I_i=[a_i, b_i]$ of the real line and a map $P:\bigcup_{I_i \in \tau}I_i \to \Lambda$ such that for any $t_1,t_2\in I_i$ we have
$$X_{t_2}(P(t_1))=X_{t_1}(P(t_2)).$$

$S$ is said to be $K$\emph{-spaced} if $a_{i+1} \geq b_i+K$ for all $i \in \{1,\cdots , m\}$ and the minimal such $K$ is called the spacing of $S$. If $\tau = \{I_1,I_2\}$ then $S$ is said to be a \emph{weak specification}.
We say that $S$ is $\eps$\emph{-shadowed} by $x\in\Lambda$ if $d(X_t(x), P(t)) < \eps$ for all $t\in \bigcup_{I_i \in \tau}I_i$. 

\begin{definition}
An invariant compact subset $\Lambda$ of $M$ has the  \emph{weak specification property} if for any $\eps > 0$ there exists a $K = K(\eps)\in \R$ such that any $K$-spaced weak specification $S$ is $\eps$-shadowed by a point of $\Lambda$. 
In this case the vector field $X|_{\Lambda}$ is said to have the weak specification property.
We say that the vector field  $X$ has the weak specification property if $M$ has it.
\end{definition}


Let $\Lambda$ be an invariant compact set. We say that $\Lambda$ is \emph{isolated in} $U$ if there is a (compact) neighborhood $U$, called an \emph{isolating block}, of $\Lambda$ such that $\Lambda = \Lambda_X(U)$. Where $$\Lambda_{X}(U) = \bigcap_{t \in \R} X_{t}(U).$$

\begin{definition}
We say that an isolated set $\Lambda$ has the \emph{ weak specification property robustly} if $\Lambda$ has an isolating block $U$ and there exists a $C^1$-neighborhood $\mathcal{U}$ of $X$ such that for any $Y \in \mathcal{U}$,  $Y|_{\Lambda_{Y}(U)}$ has the weak specification property. In this case the vector field $X|_{\Lambda}$ is said to have the weak specification property robustly.
The vector field $X$ has the weak specification property robustly if $M$ has it.
\end{definition}

We say that an isolated set $\Lambda$ is topologically mixing if for all open sets $U$ and $V$ of $\Lambda$ there is $N>0$ such that $$U\cap X_t(V)\neq\emptyset,\;\forall t\geq N.$$

Now, we give the well known notion of hyperbolic sets.

\begin{definition}
 Let $X$ be a vector field on a compact manifold $M$. An invariant and compact subset $\Lambda$ is called a hyperbolic set if there exist an invariant continuous splitting $T_{\Lambda}M=E^s\oplus\langle X\rangle\oplus E^u$ and constants $C>0$ and $\lambda>0$ such that for every $x\in \Lambda$ we have
\begin{enumerate}
\item $\|DX_tv\|\leq Ce^{-\lambda t}\|v\|$ , for every $v\in E^s_x-\{0\}$ and
\item $\|DX_{-t}v\|\leq Ce^{-\lambda t}\|v\|$ , for every $v\in E^u_x-\{0\}$.
\end{enumerate}
If the whole manifold $M$ is hyperbolic, we say that $X$ is an Anosov flow.
\end{definition}

Our main theorems in the robust context are the following.

\begin{theorem}\label{t.weakhyplocal}
If $\Lambda$ is an isolated set which has the weak specification property robustly then $\Lambda$ is a  topologically mixing  hyperbolic set.
\end{theorem}

As a consequence of these results, we obtain the following.

\begin{corollary}\label{global}
If $X$ is a vector field which has the weak specification property robustly then it generates  a topologically mixing Anosov flow.
\end{corollary}

We say that a subset $\SR\subset \Mundo$ is a residual subset if contains a countable intersection of open and dense sets. The finite intersection of residual subsets is a residual subset. Since $\Mundo$ is a Baire space when equipped with the $C^1$-topology, any residual subset of $\Mundo$ is dense.

We will say that a property holds \emph{generically} if there exists a residual subset $\SR$ such that any $X\in \SR$ has that property. Sometimes, we will say that a vector field $X$ is generic when we refer that $X$ could be taken in a residual subset. As an example, it is well known that the set of Kupka-Smale vector fields is residual in $\Mundo$, see \cite{Palis}, so we could say that generic vector fields are Kupka-Smale.

Our main result dealing with generic vector fields is the following.

\begin{theorem}\label{maintheo continuous}
There is a residual subset $\SR$ of $\Mundo$ such that if $X\in\SR$ and $X$ satisfies the weak specification property, then $X$ is Anosov.
\end{theorem}

\section{Consequences of the Weak Specification Property}

\begin{lemma}
\label{l.transitive}
Let $X$ be a vector field with an invariant compact set $\Lambda$. If $X|_{\Lambda}$ has the weak specification property then the flow on $\Lambda$ is topologically mixing.
\end{lemma}

\begin{proof}
Let $U$ and $V$ be two open sets of $\Lambda$, $x_0\in U$ and $y_0\in V$. There exists $\eps>0$ such that $B(x_0,2\eps)\subset U$ and $B(y_0,2\eps)\subset V$. The weak specification property gives us some $K>0$. Now we fix $Q>0$ and define $x=x_0$ and $y=X_{-K-Q}(y_0)$, and choose $\eta>0$ such that if $I_1=[0,\eta]$ and $I_2=[K+Q,K+Q+\eta]$ then $X_{I_1}(x)\subset B(x_0,\eps)$ and $X_{I_2}(y)\subset B(y_0,\eps)$. This gives a $K$-spaced specification, thus there exists $z$ which $\eps$-shadows this specification. By the triangle inequality, we have that $X_{K+Q}(U)\cap V\neq \emptyset$, and this holds for every $Q>0$.
\end{proof}


We define the strong stable and stable manifolds of a hyperbolic periodic point $p$ respectively as:
$$W^{ss}(p)=\{y\in M;\lim_{t\to +\infty}d(X_t(y),X_t(p))=0\}$$ and $$W^s(O(p))=\bigcup_{t\in \R}W^{ss}(X_t(p)).$$
If $\eps>0$ the local strong stable manifold is defined as
$$W^{ss}_{\eps}(p)=\{y\in M;d(X_t(y),X_t(p))<\eps \textrm{ if }t\geq 0\}.$$
By the stable manifold theorem, there exists an $\eps=\eps(p)>0$ such that
$$W^{ss}(p)=\bigcup_{t\geq 0}X_{-t}(W_{\eps}^{ss}(X_t(p))).$$
If $\sigma$ is a hyperbolic singularity of $X$ then there exists an $\eps=\eps(\sigma)>0$ such that
$$W_{\eps}^s(\sigma)=\{y\in M;d(X_t(y),\sigma)<\eps \textrm{ if }t\geq 0\}.$$ and
$$W^s(\sigma)=\bigcup_{t\geq 0}X_{-t}(W_{\eps}^{s}(\sigma)).$$
Analogous definitions holds for unstable manifolds.

\begin{lemma}\label{contraexemplo}
Let $X$ be a vector field and $\Lambda$ be an invariant compact set. If $\Lambda$ has a hyperbolic singularity $\sigma$ and for all $x \in \Lambda$, $x \in W^s(\sigma) \cap W^u(\sigma)$, then $\Lambda $ doesn't have the weak specification property.
\end{lemma}
\begin{proof}
First we will see the case where $\Lambda = \{\sigma\} \cup (O(x))$. By hypothesis $O(x) \subset W^s(\sigma) \cap W^u(\sigma)$ and is easy to see that  $X_t|_{\Lambda}$ is not a topologically mixing flow. This means that $X|_{\Lambda}$ has not the weak specification property by lemma \ref{l.transitive}.

Now we will deal with the set that has more than one regular orbit. Let $p$ and $q$ any points in $\Lambda - \{\sigma\}$ satisfying $O(p) \cap O(q) = \emptyset$ and take the ball $B(\sigma, \rho)$ with $\rho = d(\sigma , q) / 2$. Since $p \in W^s(\sigma) \cap W^u(\sigma)$ there exists $T>0$ such that for all $t > T$ we have $X_{-t}(p)$ and $X_t(p) \in B(\sigma , \rho)$. This means that $d(X_{-t}(p), q) > \rho$ and $d(X_{t}(p), q) > \rho$ for all $t > T$. The set $X_{[-T,T]}(p)$ is compact and then $d(X_{[-T,T]}(p), q) = \beta > 0$ because $O(p) \cap O(q) = \emptyset$.

This proves that there exists a positive distance between $q$ and $O(p)$, and therefore $O(p)$ is not a dense orbit in $\Lambda$. Since we take $p$ arbitrarily, we conclude that the set $\Lambda$ is not transitive for $X$ and by lemma \ref{l.transitive} $X|_{\lambda}$ has not the weak specification property as we want.
\end{proof}

The dimension of the stable manifold $W^s(O(p))$ is called the index of $O(p)$ and we denote it by $\textrm{index}(O(p))$. We remark that by hyperbolicity, if $\sigma$ is a critical hyperbolic orbit of a vector field $X$ then there exists a neighborhood $U \subset M$ of $\sigma$ and a $C^1$-neighborhood $\mathcal{U}$ of $X$ such that if $Y \in \mathcal{U}$, $Y$ has a critical hyperbolic orbit $\sigma_{Y}$ on $U$ and $\textrm{index}(\sigma) = \textrm{index}(\sigma_Y)$. Such a $\sigma_Y$ is called the continuation of $\sigma$.

\begin{theorem} \label{intersects}
Let $X$ be a  vector field with an invariant compact set $\Lambda$. If $X|_{\Lambda}$ has the weak specification property then for every two distinct hyperbolic critical orbits $O$ and $O'$ the invariant manifolds $W^u(O)$ and $W^s(O')$ intersects.
\end{theorem}

\begin{proof}
First, we deal with the case where $O=O(p)$, $O'=O(q)$ and $p$ and $q$ are periodic points. We already know that there are no sinks or sources, so $p$ and $q$ must be saddles. Let $\eps=\min\{\eps(p),\eps(q)\}$, and $K$ given by specification. If $t>0$ then take $I_1=[0,t]$ and $I_2=[K+t,K+2t]$. Now define $P(s)=X_{s-t}(p)$ if $s\in I_1$ and $P(s)=X_{s-K-t}(q)$ if $s\in I_2$. Note that this is a $K$-spaced weak specification.

So, there exists $x_t$ which shadows this weak specification:
$$d(X_s(x_t),P(s))\leq \eps\textrm{ if }s\in I_1\cup I_2.$$
Using the change of variables $u=t-s$, for every $u\in[0,t]$ we have:
$$d(X_{-u}(X_t(x_t)),X_{-u}(p))=d(X_{t-u}(x_t),X_{-u}(p))\leq \eps$$
and using $u=s-K-t$, for every $u\in [0,t]$ we have
$$d(X_u(X_{K+t}(x_t)),X_u(q))\leq \eps.$$
If $y_t=X_t(x_t)$ then we can assume that $y_t\to y$. And taking limits in the previous inequalities we obtain
$$d(X_{-u}(y),X_{-u}(p))\leq \eps\textrm{ for every $u\geq 0$, and}$$
$$d(X_u(X_{K}(y)),X_u(q))\leq \eps\textrm{ for every $u\geq 0$}.$$
The first one says that $y\in W^{uu}_{\eps}(p)\subset W^u(O(p))$ and the second one says that $X_K(y)\in W^{ss}_{\eps}(q)$, hence $y\in W^s(O(q))$.

Now, we deal with the case where $O \in Sing(X)$, $O' \in Sing(X)$ or both. In all cases we use the same proof, we just replace $X_{s-t}(p)$ and $X_{-u}(p)$ by $\sigma$ if $\sigma = O \in Sing(X)$ and $X_{s-K-t}(q)$ and $X_{u}(q)$ by $\sigma'$ if $\sigma' = O' \in Sing(X)$ and we conclude that $y \in W_{\eps}^u(\sigma) \subset W^u(\sigma)$ and $X_K(y) \in W_{\eps}^s(\sigma')$, hence $y \in W^s(\sigma')$.

\end{proof}

Now, we remark a simple property of Kupka-Smale vector fields.

\begin{lemma}\label{l.vazio}
Let $X \in \Mundo$ be a Kupka-Smale vector field and let $\sigma , \tau$ be critical hyperbolic orbits for $X$ such that $\dim W^s(\sigma) + \dim W^u(\tau) \leq \dim M$ then $W^s(\sigma) \cap W^u(\tau) = \emptyset$.
\end{lemma}
\begin{proof}
Consider first the case where  $\dim W^s(\sigma) + \dim W^u(\tau) < \dim M$. Since $X$ is a Kupka-Smale vector field, we have that $W^s(\sigma) \cap W^u(\tau) = \emptyset$ as we wanted.

Now consider the case where $\dim W^s(\sigma) + \dim W^u(\tau) = \dim M$.

Suppose  there exists $x \in W^s(\sigma) \cap W^u(\tau)$. Then $O(x) \subset W^s(\sigma) \cap W^u(\tau)$ and   we can split $$T_x(W^s(\sigma))= T_x(O(x)) \oplus E^1$$ and $$T_x(W^u(\tau))=T_x(O(x)) \oplus E^2.$$  So, $$\dim (T_x(W^s(\sigma)) + T_x(W^u(\tau)))  < \dim W^s(\sigma) + \dim W^u(\tau) = \dim M.$$
Thus $W^s(\sigma)$ is not transverse to $W^u(\tau)$ and this is a contradiction because $X$ is a Kupka-Smale vector field. This shows us that  $W^s(\sigma) \cap W^u(\tau) = \emptyset$ and proves the lemma.
\end{proof}

With this, we obtain a key consequence of theorem \ref{intersects}.

\begin{theorem}
\label{t.dicotomia}
Let $X|_{\Lambda}$ be a vector field which has the weak specification property robustly, such that any critical orbit is hyperbolic. Suppose that $Crit(X)\cap \Lambda\neq \emptyset$ then either $Crit(X) \cap \Lambda\subset PO(X)$ or $Crit(X)\cap \Lambda=\{\sigma\}$ for some singularity $\sigma\in Sing(X)$. The same holds if $X$ is a generic vector field such that $X|_{\Lambda}$ has the weak specification property.
\end{theorem}
\begin{proof}
First, let $X|_{\Lambda}$ be a vector field which has the weak specification property robustly, and let $\mathcal{U}$ be a $C^1$-neighborhood of $X$ given by the definition.

If the conclusion is false then $\Lambda$ has a hyperbolic singularity $\sigma$ with index $i$ and a distinct hyperbolic critical orbit $\tau$ with index $j$. Then there is a $C^1$-neighborhood $\mathcal{V}  \subset \mathcal{U}$ of $X$ such that for any $Z \in \mathcal{V}$, there are the continuations $\sigma_Z , \tau_Z \subset \Lambda_Z(U)$ of $\sigma$ and $\tau$  respectively.

If $j > i$ then  $\dim W^s(\sigma) + \dim W^u(\tau) \leq \dim M$. By Kupka-Smale's theorem, there exists $W \in \mathcal{V}$ such that $\dim W^s(\sigma_W) + \dim W^u(\tau_W) \leq \dim M$ and $W|_{\Lambda_W(U)}$ has the weak specification property. By Lemma \ref{l.vazio}, we have that $W^s(\sigma_W) \cap W^u(\tau_W) = \emptyset$ and this contradicts Theorem \ref{intersects}.

If $j \leq i$ then  $\dim W^u(\sigma) + \dim W^s(\tau) \leq \dim M$ and by the same arguments we have a contradiction.

In the generic case, the proof is the same, since we can assume that $X$ is Kupka-Smale.
\end{proof}

\section{Periodic Orbits}

In this section we analyze the hyperbolicity of the periodic orbits in the presence of the weak specification property. First, we show a result which its proof is  similar to the proof of theorem \ref{t.dicotomia}.

\begin{theorem}
\label{l.index}
If $X|_{\Lambda}$ is a vector field which has the weak specification property robustly, then the index of all hyperbolic periodic orbits in $\Lambda$ which are  saddles is constant and this property is robust with the same index.
\end{theorem}

\begin{proof}
Let $X|_{\Lambda}$ be a vector field which has the weak specification property robustly, and let $\mathcal{U}$ be as in the property.

Fix $Y \in \mathcal{U}$, and let $\sigma , \tau \subset \Lambda_Y(U)$ be hyperbolic periodic orbits of $Y_t$ which are saddles. Then there is a $C^1$-neighborhood $\mathcal{V}_Y \subset \mathcal{U}$ of $Y$ such that for any $Z \in \mathcal{V}_Y$, there are the continuations $\sigma_Z , \tau_Z \subset \Lambda_Z(U)$ of $\sigma$ and $\tau$  respectively.

Suppose that $\textrm{index}(\sigma) < \textrm{index}(\tau)$ (the other case is similar), then for any $Z \in \mathcal{V}_Y$ we have
$$\dim W^s(\sigma_Z) + \dim W^u(\tau_Z) \leq \dim M.$$

But since, we can take $Z$ as a Kupka-Smale vector field. By lemma \ref{l.vazio} we have that
$$W^s(\sigma_Z) \cap W^u(\tau_Z) = \emptyset.$$

On the other hand, since $Z \in \mathcal{U}$, $Z|_{\Lambda_Z(U)}$ has specification property, and this contradicts Theorem \ref{intersects}.
\end{proof}

\vspace{0.1in}

{\it Hyperbolicity of Periodic Orbits}

\vspace{0.1in}

Let $X\in\Mundo$, $p\in M$  be a  point in a periodic orbit of $X_t$ with period $T>0$ and $T_pM(s)=\{v \in T_pM; \|v\|<s\}$. Define $\langle X(p)\rangle$ as the subspace generated by $X(p)$, and set $$N_p=\langle X(p)\rangle^\perp\quad\mbox{ and }\quad N_{p,s}=N_p\cap T_pM(s),\;\mbox{ for }0<s<1$$ such that the exponential map $\exp_p:T_pM(s)\rightarrow M$ is well defined for all $p\in M$. Finally define $\Pi_{p,s}=\exp_p(N_{p,s})$. Then for a given $p'=X_{t_0}(p)$ with $t_0>0$, there are $r_0>0$ and a $C^1$ map defined as
$$
\tau:\Pi_{p,r_0}\rightarrow \R\quad\mbox{such that}\quad X_{\tau(y)}(y)\in\Pi_{p',s},$$
for all $y\in\Pi_{p,r_0}\;\mbox{with}\;\tau(p)=t_0.$

The flow $X_t$ uniquely defines the \textit{Poincar\'{e} map}
$$
\begin{array}{rcc}
f:\Pi_{p,r_0}&\longrightarrow&\Pi_{p',s}\\[0.5em]
y&\longmapsto&X_{\tau(y)}(y).
\end{array}
$$
This map is a $C^1$ embedding whose image set is contained in the interior of $\Pi_{p',s}$ if $r_0$ is small.

If $X_t(p)\not=p$ for $0<t\leq t_0$ and $r_0$ is sufficiently small, then  the map $(t,y)\mapsto X_t(y)$ is a $C^1$ embedding from the set$$\{(t,y)\in\R\times\Pi_{p,r}:0\leq t\leq\tau(y)\},$$ on $M$,  for $0<r\leq r_0$. The image is denoted by $$F_p(X_t,r,t_0)=\{X_t(y):y\in\Pi_{p,r}\;\mbox{and}\;0\leq t\leq\tau(y)\}.$$ For $\varepsilon>0$, let $V_\varepsilon(\Pi_{p,r})$ be the set of diffeomorphisms $\xi:\Pi_{p,r}\rightarrow\Pi_{p,r}$ such that $supp(\xi)\subset\Pi_{p,r/2}$ and $d_{C^1}(\xi,id)<\varepsilon$. Here $d_{C^1}$ is the usual $C^1$ metric, $id:\Pi_{p,r}\rightarrow\Pi_{p,r}$ is the identity map, and $supp(\xi)$ is the closure of the set where it differs from $id$.

Taking $p'=X_t(p)=p\;\mbox{and}\;f:\Pi_{p,r_0}\rightarrow\Pi_{p,s}$ is the \textit{Poincar\'e map}, then $f(p)=p$. In this case, the orbit of $p$,  $O(p)$ , is hyperbolic if and only if $p$ is a hyperbolic fixed point of $f$.

\begin{lemma}\label{b}
Let $X\in \Mundo,\;p$  be a  periodic orbit of $X_t$ with period $T>0$, let $f:\Pi_{p,r_0}\rightarrow \Pi_{p,s}$ be as above, and let $\mathcal{U}\subset \Mundo$ be a  $C^1$ neighborhood of $X$ and $0<r\leq r_0$ be given. Then there are $\delta_0>0$ and $0<\varepsilon_0<r/2$ such that for a linear isomorphism $H_\delta:N_p\rightarrow N_p$ with $||H_\delta- D_pf||<\delta<\delta_0$, there is $Y^\delta\in \mathcal{U}$ satisfying:
\begin{enumerate}
\item[(i)] $Y^\delta(x)=X(x)$, if $x\not\in F_p(X_t;r;T)$,

\item[(ii)] $p$ belongs to a periodic orbit for $Y_t^\delta$,

\item[(iii)] \begin{displaymath}g_{Y^\delta}(x)=\left\{ \begin{array}{ll}
 \exp_p\circ H_\delta\circ\exp_p^{-1}(x), & \mbox{if}\;x\in B_{\varepsilon_0/4}(p)\cap\Pi_{p,r}\\[0.5cm]
f(x), & \mbox{if}\;x\not\in B_{\varepsilon_0}(p)\cap\Pi_{p,r},
\end{array} \right.
\end{displaymath} where $g_{Y^\delta}:\Pi_{p,r}\rightarrow\Pi_{p,s}$ is the Poincar\'{e} map of $Y^\delta_t$. Furthermore, let $Y^0$ be the vector field for $H_0=D_pf$. Then we have

\item[(iv)] $d_{C_0}(Y^\delta,Y^0)\rightarrow 0$ as $\delta\rightarrow 0$.
\end{enumerate}
\end{lemma}

\begin{proof}
See \cite[Lemma 1.3 (pg. 3395)]{MSS01}.
\end{proof}

This lemma allows us to find a vector field $Y$ sufficiently close to $X$ whose Poincar\'e map at $p\in Per(Y)$ is a perturbation of the derivative of the Poincar\'e map at $p\in Per(X)$.

\begin{theorem}
\label{l.nonhyp}
 Let $X$ be a vector field, \  $\Lambda$ be an isolated set with the weak specification property robustly, \  and $\SU$ be a $C^1$-neighborhood of $X$. If a periodic orbit of $X$ in $\Lambda$ is not hyperbolic then there exists a vector field $Y\in \SU$ with two hyperbolic periodic orbits in $\Lambda_Y(U)$ with different indices.
\end{theorem}

\begin{proof}

We will make the demonstration by contradiction.

Let $U$ be an isolated block for $\Lambda$, $\mathcal{V}\subset\SU$ be a $C^1$ neighborhood of $X$ such that for every vector field $Y$ in $\mathcal{V}$ the continuation of $\Lambda$, $$\Lambda_Y(U)=\bigcap_{t\in \R}(Y_t(U)),$$ has the weak specification property robustly, and $\sigma$ be a periodic orbit of $X$ in $\Lambda$ which is not hyperbolic. Take $p\in\sigma$, and set $T>0$ the period of $p$. Let $r_0>0$ and $f:\Pi_{p,r_0}\rightarrow \Pi_p$ the Poincar\'e map for $X$. As $\sigma$ is not hyperbolic, then $D_pf$ admits an eigenvalue $\lambda$ with $|\lambda|=1$. Let $E$ be an eigenspace associated to $\lambda$ such that the dimension of $E$ is either $1$, if $\lambda\in\R$, or $2$, if $\lambda\in\mathbb C\backslash\R$.

Let $E$ be an eigenspace associated to $\lambda$ such that the dimension of $E$ is either $1$, if $\lambda\in\R$, or $2$, if $\lambda\in\mathbb C\backslash\R$. Also, let $F_0$ be the subspace of $N_p$ consisting of all eigenvectors such that the eigenvalue associated has norm equals to one.
Then there is $G$ a subspace of $N_p$ such that $N_p=F_0\oplus G$. Also, there is $F$ an subspace of $F_0$ such that $F_0=E\oplus F$.
Then $N_p=E\oplus F\oplus G$ can be written as the direct sum of an eigenspace associated to $\lambda$ with minimum dimension, the subspace generated by all other eigenvalues of norm one and a subspace without eigenvectors of norm one.

Now, using Lemma \ref{b}, we will find a vector field $Y,\;C^1$-close to $X$ such that $\lambda$ is the only eigenvalue with $|\lambda|=1$.

Let $\delta_0=\delta_0(X)>0$ and $0<\varepsilon_0=\varepsilon_0(X)<r/2$ given by Lemma \ref{b}. Take $0<\delta<\delta_0$ and $H_\delta:N_p\rightarrow N_p$ a linear isomorphism such that \begin{displaymath}\begin{array}{l}
H_\delta(v)=D_pf(v), \mbox{ for all } v\in E\oplus G;\\
H_\delta(v)=(\delta)v+D_pf(v), \mbox{ for all } v\in F.\end{array}
\end{displaymath}
Then $||H_\delta-D_pf||<\delta$ and, decreasing $\delta$ if necessary, we can assume that $\lambda$ is the only eigenvalue of $H_\delta$ with norm equals to one.

By Lemma \ref{b}, there is a vector field $Y\in\mathcal V$, and therefore  such that the continuation $\Lambda_Y(U)$ has the weak specification property robustly, such that $p$ belongs to a periodic orbit for $Y_t$, and
\begin{displaymath}g_Y(x)=\left\{ \begin{array}{ll}
 \exp_p\circ H_\delta\circ\exp_p^{-1}(x), & \mbox{if}\;x\in B_{\varepsilon_0/4}(p)\cap\Pi_{p,r}\\[0.5cm]
f(x), & \mbox{if}\;x\not\in B_{\varepsilon_0}(p)\cap\Pi_{p,r},
\end{array} \right.
\end{displaymath} where $g_Y:\Pi_{p,r}\rightarrow\Pi_{p,s}$ is the Poincar\'{e} map of $Y_t$.

As $p\in U$ belongs to a periodic orbit for $Y_t$, then $p$ belongs to continuation $$\Lambda_Y(U)=\bigcap_{t \in \R}(Y_t(U)).$$ Also, as $D_pg_Y=H_\delta$, then $\lambda$ is  the only eigenvalue of $D_p{g_Y}$ with modulus one.

Denote by $E^s_p$ the subspace of $N_p$ associated to the eigenvalues  of $D_pg_Y$  with modulus less than one, by $E^u_p$ the subspace associated to the eigenvalues of $D_pg_Y$ with modulus greater than one, and by $E^c_p$ the subspace of $D_pg_Y$ associated to $\lambda$. Then $$T_p\Pi_{p,r}=E^s_p\oplus E^c_p\oplus E^u_p.$$ Let be $\mathcal{V}_Y$ a $C^1$ neighborhood of $Y$  such that every vector field in $\mathcal{V}_Y$ is such that the continuation of $\Lambda_Y(U)$ has the weak specification property.

Again, with help of Lemma \ref{b}, we will construct two hyperbolic periodic orbits with different indices.

If $\dim E^c_p=1$:

We suppose that $\lambda=1$ and the other case is similar.

Then we have $$g_Y(x)=x,\quad\forall x\in \exp_p(E^c_p)\cap B_{\varepsilon_0/4}(p)\cap\Pi_{p,r_0}.$$ As $x$ is a fixed point for $g_Y$, the orbit of $x$ is periodic with period $T_x>0$  and, decreasing $r_0$ if necessary, we can assume $x\in U$, hence, $x\in\Lambda_Y(U)$. Note that all $x\in \exp_p(E^c_p)\cap B_{\varepsilon_0/4}(p)\cap\Pi_{p,r_0}$ is a non-hyperbolic fixed point for $g_Y$.

 Fix $q\in\exp_p(E^c_p)\cap B_{\varepsilon_0/4}(p)\cap\Pi_{p,r_0}$, and take $r>0$ such that $$F_p(Y_t;r;T)\cap F_q(Y_t;r;T_q)=\emptyset.$$

By continuity of the derivative of the flow $Y_t$, taking $q$ closer to $p$ if necessary, we can assume that $\textrm{index}(O(p)) = \textrm{index}(O(q)) = s$.

Using the Lemma \ref{b}, we will make a perturbation in the vector field $Y$ in a small neighborhood of the orbit of $p$, yielding a hyperbolic periodic orbit of index $s+1$.

Let $\mathcal{V}_Y$ be a $C^1$-neighborhood of $Y$ such that for every $Z\in \mathcal{V}_Y$, the continuation $\Lambda_Z(U)$ has the weak specification property robustly. By Lemma \ref{b}, there are $0<\varepsilon_0=\varepsilon_0(Y)<r/2$ and $\delta_0=\delta_0(Y)$ such that the lemma holds. Take $A:N_p\rightarrow N_p$ a hyperbolic linear isomorphism such that $$A(v)=v\;\mbox{if}\;v\not\in E,\quad\mbox{and}\quad A(v)=(1-\eta) v\;\mbox{if}\;v\in E,$$where $0<\eta<\delta_0$. Then $$||A-D_pg_Y||<\delta_0\quad\mbox{ and }\quad\dim E_A=s+1,$$ where $E_A$ is the eigenspace associated to the eigenvalues of $A$ with modulus less than one.

Then, there is $Z\in \mathcal{V}_Y$ for which the orbit of $p\in U$ is still a periodic orbit, therefore $p\in\Lambda_Z(U)$, $Z(x)=Y(x)$, if $x\not\in F_p(Y_t;r;T)$, and
\begin{displaymath}g_{Z}(x)=\left\{ \begin{array}{ll}
\exp_p\circ A\circ\exp_p^{-1}(x), & \mbox{if}\;x\in B_{\varepsilon_0(Y)/4}(p)\cap\Pi_{p,r}\\[0.5cm]
g_Y(x), & \mbox{if}\;x\not\in B_{\varepsilon_0(Y)}(p)\cap\Pi_{p,r},
\end{array} \right.
\end{displaymath} where $g_Z:\Pi_{p,r}\rightarrow\Pi_p$ is the Poincar\'{e} map of $Z_t$.

Then, $D_pg_Z$ is equivalent to $A$ and therefore, $p$ is a hyperbolic periodic point of $Z_t$. As $p\in U$, then $p$ belongs to the continuation $\Lambda_Z(U)$ of $\Lambda_Y(U)$  and, for $Z$, $\textrm{index}(p)=\dim E_A=s+1$.

On the other side, as $q\not\in F_p(Y_t;r;T)$, then $Z(x)=Y(x)$ in a neighborhood of $q$, and therefore, the $Z_t$-orbit of $q$ is a non-hyperbolic periodic orbit such that $\textrm{index}(O(q))=s$. Also, as $q\in U$, then $q\in\Lambda_Z(U)$.

Now, making an analogous  perturbation, we can find a vector field $\overline{Z} \in \SV_Z$  such that the $\overline{Z}_t$-orbit of $q$ is a  hyperbolic periodic orbit,  $\textrm{index}(O(q))=s$ and the continuation $\Lambda_{\overline{Z}}(U)$ of $\Lambda_Z(U)$ has the weak specification property robustly.

So, we have found a vector field $\overline{Z}$ arbitrarily close to $X$ such that the continuation $\Lambda_{\overline{Z}}(U)$ has the weak specification property robustly and there are $p,q \in \Lambda_{\overline{Z}}(U)$ hyperbolic periodic points such that $\textrm{index}(O(p)) \neq \textrm{index}(O(q))$ as we wanted.

If $\dim E^c_p=2$:



Then $D_pg_Y$ acts in $E^c_p$ as a rotation. If this is a rational rotation there is $l>0$ such that $D_pg^l_Y(v)=v$ for all $v \in E^c_p\cap \exp_p^{-1}(B_{\varepsilon_0}(p))$. Fix $l$ the minimum with this property. As in the previous case, with a $C^1$-modification we can find a vector field $Z$ in $\mathcal{V}_Y$  such that the continuation  $\Lambda_Z(U)$ has the weak specification property and it has two periodic orbits with different indices.

If the rotation is irrational there is a linear isomorphism $A:N_p\rightarrow N_p$ such that $||A-D_pg_Y||<\delta_0(Y)$ and $A$ acts in $E^c_p$ as a rational rotation. Then there is $Z\in \mathcal{V}_Y$  such that the continuation  $\Lambda_Z(U)$ has the weak specification property robustly and $D_pg^l_Z(v)=v$ for all $v \in E^c_p\cap\exp_p^{-1}( B_{\varepsilon_0}(p))$, with $l>0$ the minimum with this property. Again, we can find a vector field in $\mathcal{V}_Z$ such that the continuation of $\Lambda_Z(U)$ has the weak specification property and it has two periodic orbits with different index.
\end{proof}

In the case where the flow $X$ has the weak specification property robustly, i.e., $\Lambda=M$, as a consequence of the previous theorem we have that there is a $C^1$ neighborhood of $X$ such that all periodic orbits of every vector field in this neighborhood are hyperbolic with same index. Moreover, we obtain the following corollary.

\begin{corollary}
\label{c.periodichyp}
If $X|_{\Lambda}$ is a vector field which has the weak specification property robustly, then all periodic orbits are hyperbolic.
\end{corollary}

\begin{proof}
If not, the previous theorem give us two periodic orbits with different indices. But this is is a contradiction with theorem \ref{l.index}
\end{proof}

\section{Consequences of Genericity}

In this section, we collect some results using standard generic arguments.

\begin{lemma}\label{dif.index.cont}
There exists a residual subset $\SR$ of $\Mundo$ such that if $X\in\SR$ is $C^1$-approximated by $\{X_n\}_{n\in\N}$ such that each $X_n\in\Mundo$ has two distinct hyperbolic periodic orbits, $q_n,t_n\in Per_h(X_n)$, with different indices and with $d(q_n,t_n)<\varepsilon$, then there exists two distinct hyperbolic periodic points, $q,t \in Per_h(X)$, with different indices and with $d(q,t)<2\varepsilon$.
\end{lemma}
\begin{proof}
We can take $\{\SV_l\}_{l\in\N}$ a countable basis of open sets of $M$, and define the set
\begin{eqnarray*}
H_{l,m}(\varepsilon) \ = \ \Big\{Y\in\Mundo:&\hspace*{-.2cm}\mbox{there are }q\in\SV_l, t\in\SV_m,\mbox{ in distinct periodic hyperbolic  }\\
&\hspace*{-.2cm}\mbox{ orbits of } Y\mbox{ with different index and with }d(q,t)<\varepsilon\Big\}.
\end{eqnarray*}

Then $H_{l,m}(\varepsilon)$ is an open subset of $\Mundo$, and defining $$N_{l,m}(\varepsilon)=\Mundo-\overline{H_{l,m}(\varepsilon)},$$
we have that $H_{l,m}(\varepsilon)\cup N_{l,m}(\varepsilon)$ is an open and dense subset of $\Mundo$, which means that $$\SR(\varepsilon)=\bigcap_{l,m\in\N}H_{l,m}(\varepsilon)\cup N_{l,m}(\varepsilon)$$ is a residual subset of $\Mundo$. Therefore $$\SR=\bigcap_{r>0,\;r\in\Q}\SR(r)$$ is also a residual subset of $\Mundo$.

Let $X\in\SR$ and $X_n\in\Mundo$ such that $X_n\stackrel{C^1}{\longrightarrow}X$ as $n$ goes to infinity, and let $q_n$ and $t_n$ in distinct hyperbolic periodic orbits of $X_n$ with different index and with $d(q_n,t_n)<\varepsilon$.

From the compactness of $M$ we get, unless of a subsequence, the points: $$\overline{t} = \lim_{n \to \infty}t_n \ \  \textrm{and} \ \ \overline{q} = \lim_{n \to \infty}q_n.$$ Since $\{\SV\}_{n \in \N}$ is a basis for $M$, there exist $l,m \in \N$ such that $\overline{q} \in \SV_l$ and $\overline{t} \in \SV_m$. Thus, for a sufficiently bigger $n$, we have that $q_n \in \SV_l$ and $t_n \in \SV_m$.

We can take $\varepsilon<r<2\varepsilon$ with $r\in\Q$ and then $X\in\overline{H_{l,m}(r)}$. As $X\in\SR$ we have that $X\in H_{l,m}(r)\subset H_{l,m}(2\varepsilon)$ and, therefore, there are $q,t$ in distinct hyperbolic periodic orbits of $X$ with different index and $d(q,t)<2\varepsilon$. Moreover, $\overline{q}=q$ and $\overline{t} = t$.
\end{proof}

Recall that if $p$ belongs to a periodic orbit of $X$, then $DX_{T(p)}(p)$ has 1 as eigenvalue with eigenvector $X(p)$, and all the other eigenvalues are called the characteristic multipliers of $p$.

\begin{definition}
We say that a point $p$ in a hyperbolic periodic orbit of $X$ has a $\delta$-weak hyperbolic eigenvalue if there is a characteristic multiplier $\sigma$ of the orbit of $p$ such that $$(1-\delta) < |\sigma| < (1+\delta).$$
\end{definition}

The next lemma  was proved by the second author in \cite{Laura}, here we will give a sketch of the proof.

\begin{lemma}\label{whe.cont}
There exists a residual subset $\SR_0$ of $\Mundo$ such that if $X\in\SR_0$ is $C^1$-approximated by $\{X_n\}_{n\in\N}$ such that  there exists at least one point in each $Per_h(X_n)$ with $\delta$-whe then there exists  a point in $Per_h(X)$ with  $2\delta$-whe.
\end{lemma}
\begin{proof}
We use the same idea of the proof of the lemma \ref{dif.index.cont}. Take $$H_n(\delta)= \{Z \in \Mundo : \textrm{there} \ \textrm{is} \  p \in \SV_n \cap Per_h(Z) \ \textrm{with} \ \delta \textrm{-whe} \}.$$
We claim that if $Z\in H_n(\delta)$, the fact that the orbit of $p$ is hyperbolic implies that every $Y\in\Mundo$ sufficiently close to $Z$ admits a continuation of $p$ in $Per_h(Y) \cap \SV_n$ with a $\delta$-whe. This means that $H_n(\delta)$ is an open set.

Suppose not. Then there exists a sequence $\{Y_k\in\Mundo\}_{k\in\N}$ that converges to $Z$ such that every point in $\SV_n \cap Per_h(Y_k)$ has no $\delta$-whe. By continuity of the flow and its derivative, we can assume that, for sufficiently big $k$,  the continuation $p_k$ of $p$ has an eigenvalue $\sigma_k$ sufficiently close to the $\delta$-whe of $p$, $\sigma$. Thus we have that $|\sigma_k| \geq (1 + \delta)$ or $|\sigma_k| \leq (1 - \delta)$.  We will assume the first case, and the second one is analogous. Let $\varepsilon>0$, then we can take $k$ big enough so $|\sigma-\sigma_k|<\varepsilon$. Then $$(1+\delta)\leq |\sigma_k| \Rightarrow (1+\delta)<|\sigma|+\varepsilon.$$ As $\varepsilon\rightarrow0$ we have $|\sigma|\geq(1+\delta)$, a contradiction.

Now, repeating the argument given in proof of lemma \ref{dif.index.cont} we end this proof.

\end{proof}

\begin{proposition}\label{no.whe.cont}
There exists a residual subset $\SR$ of $\Mundo$ such that if $X\in \SR$ and $\Lambda$ is a compact invariant set for $X$ with the weak specification property then there exist $\delta>0$ such that no point in $\Lambda$  has a $\delta$-whe.
\end{proposition}

\begin{proof}
Let $\SR$ be the residual given by the intersection of the residual given by Lemma \ref{dif.index.cont} and the set of Kupka-Smale vector fields.

Suppose, by contradiction, that there exists $p_n\in Per(X)\cap\Lambda$ such that $p_n$ has a $1/n$-whe. Then by Lemma \ref{b} with similar arguments to the proof of the Theorem \ref{l.nonhyp} we can see that there exist $X_n\stackrel{C^1}{\longrightarrow} X$ with two periodic orbits $q_n,t_n\in\Lambda$ such that $ind(q_n)\neq ind(t_n)$ and $d(t_n,q_n) < \epsilon$. By Lemma \ref{dif.index.cont}, $X$ itself has two periodic orbits $q\neq t$, both in $\Lambda$ and with different index. By Lemma \ref{intersects} the invariant manifolds of $p$ and $q$ have an intersection which is not transversal, since they have different index. This is a contradiction with the fact that $X$ is Kupka-Smale.
\end{proof}

\section{Singularities}




In the presence of weak specification, we can obtain analogous results, as in the previous section, for singularities. 



First, we recall the following tool.

\begin{lemma}\label{a}
Let $X\in \Mundo$ and $p\in Sing(X)$. Then for every $C^1$-neighborhood
$\mathcal{U} \subset \Mundo$ of X, there are $\delta_0 > 0$ and $\varepsilon_0 > 0$ such that if $H_\delta :T_pM\rightarrow T_pM$ is a linear map with $||H_\delta - D_pX||< \delta<\delta_0$, then there is $Y^\delta\in \mathcal{U}$ satisfying
\begin{displaymath}
Y^\delta(x) =
\left\{ \begin{array}{ll}
(D_{\exp^{-1}_p(x)}\exp_p)\circ H_\delta\circ \exp_p^{-1}(x), & \mbox{if}\;x\in B_{\varepsilon_0/4}(p)\\[0.5cm]
X(x), & \mbox{if}\;x\not\in B_{\varepsilon_0}(p).
\end{array} \right.
\end{displaymath}
Furthermore,
$$
d_{C_0}(Y^\delta, Y^0)\longrightarrow 0\quad\mbox{as}\quad\delta\longrightarrow 0.
$$
Here $Y^0$ is the vector field for $H_0= D_pX$.
\end{lemma}

\begin{proof}
See \cite[Lemma 1.1 (pg. 3394)]{MSS01}.
\end{proof}

This lemma allows us to find a vector field $Y=Y^0$ sufficiently close to $X$ such that $Y\mid_{B_{\varepsilon_0}/4(p)}$ is a linearization of $X\mid_{B_{\varepsilon_0/4}(p)}$ with respect to the exponential coordinates. This means that if there are an interval $I\in\R$ and an integral curve $\xi(t)\;(t\in I)$ of the linear vector field $D_pX$ in $\exp_p^{-1}(B_{\varepsilon_0/4}(p))\subset T_pM$, then the composition $\exp_p\circ\xi:I\rightarrow M$ is an integral curve of $Y$ in $B_{\varepsilon_0/4}(p)\subset M$. Moreover $D_pY^\delta$ is equivalent to $H_\delta$.

\begin{theorem}
\label{sing.nonhyp}
Let $X$ be a vector field, $\Lambda$ be a compact invariant set with the weak specification property robustly. If $\Lambda \cap Crit(X) \neq \emptyset$ then either $\Lambda$ has a unique singularity $\sigma$ which is hyperbolic or $\Lambda$ has no singularities and all periodic orbits are hyperbolic.
\end{theorem}

\begin{proof}

Let $\mathcal{V}\subset\SU$ be a $C^1$ neighborhood of $X$ such that for every vector field in $\mathcal{V}$ the continuation of $\Lambda$ has the weak specification property robustly. Suppose that $\Lambda$ has a non-hyperbolic singularity $\sigma$  of $X$. So, $D_pX$ admits an eigenvalue $\lambda$ with $Re(\lambda)=0$.


By Lemma \ref{a}, with a small modification of the linear isomorphism $D_\sigma X$, we can find a vector field $Y\in \mathcal{V}$, $C^1$-close to $X$ such that the continuation $\Lambda_Y(U)$ has the weak specification property robustly, and such that $\sigma$ is still a singularity in $\Lambda_Y(U)$, and $\lambda$ is the only eigenvalue of $D_\sigma Y$ with $Re(\lambda)=0$.

Denote by $E^s_\sigma $ the eigenspace of $D_\sigma Y$ associated to the eigenvalues with real part less than zero, by $E^u_\sigma $ the eigenspace of $D_\sigma Y$ associated to the eigenvalues with real part greater than zero, and by $E^c_\sigma $ the eigenspace of $D_\sigma Y$ associated to $\lambda$. Then $$T_\sigma M=E^s_\sigma \oplus E^c_\sigma \oplus E^u_\sigma. $$ Let $\mathcal{V}_Y$ a $C^1$ neighborhood of $Y$ such that every vector field in $\mathcal{V}_Y$ is such that the continuation of $\Lambda_Y(U)$ has the weak specification property.

If $\dim E^c_\sigma=1$:

Then $\lambda=0$ and there exists  $r>0$ such that for all $v\in E^c_\sigma(r)= E^c_{\sigma} \cap T_{\sigma}M(r)$, $Y(\exp_\sigma (v))=0$. Taking $p\in\exp_\sigma (E^c_\sigma(r)) - \{\sigma\}$,   we have that $p$ is a non-hyperbolic singularity for $Y$ and taking $p$ sufficiently close to $\sigma$, we can assume that  $\textrm{index}(\sigma) = \textrm{index}(p) =s$.

Taking $0<\varepsilon_0<d(\sigma,p)/2$ in Lemma \ref{a}, we can take $0<\delta<\delta_0$ and define a linear map $H:T_\sigma M\rightarrow T_\sigma M$ such that $$ H(v)=-\delta v,\forall v\in E^c_\sigma,\mbox{ and }H(v)=D_\sigma Y(v), \forall v\in E^s_\sigma \oplus E^u_\sigma.$$ By Lemma \ref{a} there is $Z\in \mathcal{V}_Y$ such that the continuation $\Lambda_Z(U)$ has the weak specification property and

$$Z(x)=(D_{\exp^{-1}_\sigma (x)}\exp_\sigma )\circ H\circ \exp_\sigma ^{-1}(x),\;\mbox{if}\;x\in B_{\varepsilon_0/4}(\sigma ).$$

Then we have that $\sigma$ is a hyperbolic singularity for $Z$ and $\textrm{index}(\sigma)=s+1$. As $Z(x)=Y(x)$ for all $x\not\in B_{\varepsilon_0}(\sigma)$, we have that $p$ is a non-hyperbolic singularity for $Z$ in $\Lambda_Z(U)$  with $\textrm{index}(p)=s$.

Now we will use Lemma \ref{a} again. But this time we define a linear map $\overline{H}:T_pM\rightarrow T_pM$ in order to find a vector field $\overline{Z}$ close to $Z$ such that, for some $\overline{\varepsilon}_0$,
$$\overline{Z}(x)=(D_{\exp^{-1}_p (x)}\exp_p )\circ \overline{H}\circ \exp_p ^{-1}(x),\;\mbox{if}\;x\in B_{\overline{\varepsilon}_0/4}(p ),$$  $\overline{Z}(x) = Z(x)$ if $x \not\in  B_{\overline{\varepsilon}_0}(p )$, the continuation $\Lambda_{\overline{Z}}(U)$ has the weak specification property and $p$ is a hyperbolic singularity for $\overline{Z}$ with $\textrm{index}(p)=s$. Thus we find a vector field $\overline{Z} \in \SU$ with two hyperbolic singularities in $\Lambda_{\overline{Z}}(U)$ and this contradicts theorem \ref{t.dicotomia} applied to $Z$.

If $\dim E^c_\sigma=2$:

In this case, there are no singularities of $Y$ besides $\sigma$ in the neighborhood of $\sigma$. In fact, for any $p \in \exp_{\sigma}(E_{\sigma}^c(r)) - \{\sigma\}$ $r > 0$, the $Y_t$-orbit $O_Y(p)$ of $p$ is a periodic orbit and $O_Y(p) \subset \exp_{\sigma}(E^c_{\sigma}(r))$.

If $d(\sigma ,p) = s$, let $s_0 < \min\{r-s , s\}$. Then there exists $0 < r_0 \leq s_0$ such that $f: \Pi_{p,r_0} \to \Pi_{p,s_0}$ is a Poincar\'e map. Since $E^c_{\sigma}(r)$ contains the flow direction and it is 2-dimensional, we have that $E^c_{\sigma} \cap \Pi_{p,r_0}$ is a 1-dimensional set and all $x \in E^c_{\sigma} \cap \Pi_{p,r_0}$ is a fixed point of $f$. This implies that $p$ is not a hyperbolic fixed point for $f$.

Taking $p$ closer to $\sigma$ if necessary, we can assume that $\textrm{index}(\sigma) = \textrm{index}(O(p)) = i$ and $p \in \Lambda_Y(U)$.

Using  similar arguments  to the previous case and to the proof of Theorem \ref{l.nonhyp} we can find $Z \in \SV_Y$ such that the continuation $\Lambda_Z(U)$ has the weak specification property robustly, $O(p)$ is a hyperbolic periodic $Z_t$-orbit, $\sigma$ is a hyperbolic singularity for $Z$ but this contradicts theorem \ref{t.dicotomia} applied to $Z$.

If $Crit(X|_{\Lambda})$ is not a unique singularity, then theorem \ref{t.dicotomia} says that $X$ has no singularities and it is sufficient to apply corollary \ref{c.periodichyp}

\end{proof}

\section{Proofs for The Robust Case}

We will begin given the proof of theorem \ref{t.weakhyplocal}. Let $\Lambda$ be a compact and $X$-invariant subset which has the weak specification property robustly. Using lemma \ref{l.transitive}, there are no sources nor sinks. Moreover, using  theorem \ref{sing.nonhyp}, we obtain that critical orbit is hyperbolic.

Now, we give the notion of sectional hyperbolic sets, which is weaker than hyperbolicity, see \cite{MM} for more details.

\begin{definition}
Given a vector field $X$ we say that an invariant compact set $\Gamma$ is \emph{sectional-hyperbolic} if every singularity in $\Gamma$ is hyperbolic and there exists a continuous invariant splitting $T_{\Gamma}M=E\oplus F$ over $\Gamma$ and constants $C > 0$ and $\lambda > 0$ such that for every $x\in \Gamma$ and $t\geq0$:
\begin{itemize}
\item[(i)] The splitting is not trivial: $E_x\neq0$ and $F_x\neq0$.
\item[(ii)] The splitting is dominated: $||DX_t\mid E_x||.||DX_{-t}\mid F_{X_t(x)}|| < Ce^{-\lambda t} .$
\item[(iii)] The subbundle $E$ is contracting: $||DX_t(x)v||\leq Ce^{-\lambda t}\|v\|$ , for every $v\in E_x-\{0\}$.
\item[(iv)] The subbundle $F$ is sectionally expanding: For every 2-plane section $L\subset F$, if we denote $L_x\subset F_x$ the 2-plane in the
subspace $F_x$ then $$|\det(DX_t(x)\mid_{L_x})|>Ce^{\lambda t}.$$
\end{itemize}
\end{definition}

Now, we recall the notions from \cite{Gan-Wen-Zhu}.

\begin{definition}
An invariant set $\Gamma$ is strongly homogeneous of index $i\in[0,d-1]$, where $d$ is the dimension of $M$ if there exist neighborhoods $\SU$ of $X$ and $U$ of $\Gamma$ such that for every $Y\in \SU$ and any periodic orbit of $Y$ in $U$ has index $i$.
\end{definition}

Then we invoke a theorem due to Gan, Wen and Zhu \cite{Gan-Wen-Zhu} and of Metzger and Morales (see theorem A of \cite{MM}).

\begin{theorem}
\label{t.homo}
Let $\Gamma$ be a robustly transitive set of $X$ which is strongly homogeneous of index $i$. If all singularities in $\Gamma$ are hyperbolic then all singularities in $\Gamma$ must have the same index and $\Gamma$ is sectional hyperbolic.
\end{theorem}

This implies that $\Lambda$ is sectional hyperbolic. Now using theorem \ref{t.dicotomia}, either $\Lambda$ has only hyperbolic periodic orbits and no singularities or it has only one singularity and no periodic orbits. In the former case this implies that $\Lambda$ is a hyperbolic set, using the hyperbolic lemma (see \cite{AP} or \cite{BM}).

\begin{lemma}[Hyperbolic lemma]
\label{l.hyp}
Any sectional hyperbolic set without singularities is hyperbolic.
\end{lemma}

In the latter case, we will use an standard application of Hayashi's connecting lemma \cite{H}. Actually, we will need the following version of the $C^1$ connecting lemma which can be found in \cite{Gan-Wen}:




\begin{theorem}\label{tunado}
Let $X \in \Mundo$, and $z \in M$ be neither singular nor periodic of $X$. Then for any $C^1$-neighborhood $\SU$ of $X$ in $\Mundo$, there exist $\rho > 1$, $T > 1$ and $\delta_0 > 0$ such that for any $0 < \delta \leq \delta_0$ and any two points $x, y$ outside the tube $\Delta = \bigcup_{t\in[0,T]}B(X_t(z), \delta)$, if the positive $X$-orbit of $x$
and the negative $X$-orbit of $y$ both hit $B(z, \delta /\rho)$, then there exists $Y \in \SU$ with $Y = X$ outside $\Delta$ such that $y$ is on the positive $Y$-orbit of $x$. Moreover, the resulted $Y$-orbit segment from $x$ to $y$ meets $B(z, \delta)$.
\end{theorem}

Let us suppose that $\Lambda$ does not reduces to the singularity $\sigma$. Let $\SU$ be a $C^1$-neighborhood of $X$ such that for all $Y \in \SU$, $\Lambda_Y(U)$ has the weak specification robustly and there exists $\sigma_Y$, the continuation of $\sigma$. By lemma \ref{contraexemplo} we can take $x\in (\Lambda-\{\sigma\})$ such that, either $x \notin W^s(\sigma)$ or $x \notin W^u(\sigma)$. We will deal with the first case, the second is analogous.

Since $\Lambda$ is compact, we have that $\omega(x) \neq \emptyset$ and since $x \notin W^s(\sigma)$ there exists $z \in (\omega(x) - \{\sigma\})$. Then $z$ is not a singularity nor a periodic point by theorem \ref{t.dicotomia}.

Since $z \in \omega(x)$, there exist a sequence $(t_n)_{n\in\N}$ such that $\lim_{t_n \to \infty}X_{t_n}(x) = z$. Let $0 < \delta \leq \delta_0$ such that $X_{\overline{t}}(z) \notin B(z, \delta)$ for some $\overline{t} > 0$. By the continuity of the flow we can find $m \in \N$ bigger enough such that there is $l \in (t_m,t_{m+1})$ satisfying $X_{t_m}(x) \in B(z, \delta)$, $X_{l}(x) \notin B(z, \delta)$ and $X_{t_{m+1}}(x) \in B(z, \delta)$. So we can choose $l_1< l_2$ with $l_1, l_2 \in (t_m,t_{m+1})$ such that $X_{l_1}(x)$ and $X_{l_2}(x)$  are not in $\bigcup_{t\in[0,T]}B(X_t(z), \delta) = \Delta$. Denote by $w$ and $y$, $X_{l_1}(x)$ and $X_{l_2}(x)$ respectively and note that $y$ is on the positive $X$-orbit of $w$.

By theorem \ref{tunado} there exists $Y \in \SU$ with $Y = X$ outside $\Delta$ such that $w$ is on the positive $Y$-orbit of $y$. If we take $\Delta$ such that $\Delta \subset U$, we have that the periodic orbit $O_Y(w)$ is on $\Lambda_Y(U)$. But then $\Lambda_Y(U)$ has the singularity $\sigma_Y$ and a periodic orbit, contradicting theorem \ref{t.dicotomia}.

Thus $\Lambda=\{\sigma\}$ and this finishes the proof of theorem \ref{t.weakhyplocal}.




\section{Proofs on the Generic Case}

In this section we prove theorem \ref{maintheo continuous}.

We recall the notion of star flows. Let $\mathcal{G}^1(M)$ be the set of $C^1$ vector fields in $M$ for which there is a neighborhood $\mathcal{U}$ in the $C^1$ topology such that every critical orbit of every vector field in $\mathcal{U}$ is hyperbolic. Following the literature, we say that the generated flow by $X\in \SG^1(M)$ is an star flow. Just for simplicity, we will say that $X$ is an star flow.

First, we observe that a generic vector field with the weak specification property is an star flow. Suppose, by contradiction, that $X\not\in\SG^1(M)$. Then there exists a sequence $X_n\stackrel{C^1}{\longrightarrow}X$ such that every $X_n$ admits a non-hyperbolic singularity or periodic orbit. This non-hyperbolic singularity or periodic orbit has a $\frac\delta2$-whe, by definition. Then, Lemma \ref{whe.cont} implies that $X$ has a $\delta$-whe, which is a contradiction with Proposition \ref{no.whe.cont}.


Now we proceed as follows. First, by lemma \ref{l.transitive} we have that $M = \Omega(X)$ and then by Pugh's General Density theorem \cite{Pugh} we obtain a dense set of periodic orbits. In particular, since we can also assume that $X$ is Kupka-Smale, by theorem \ref{t.dicotomia}, we have that $X$ has no singularities.

To conclude the proof we invoke Gan-Wen's theorem.

\begin{theorem}[Gan-Wen]
\label{Gan-Wen}
Let $M$ be a closed manifold. If $X$ belongs to $\mathcal G^1(M)$ and has no singularities then $X$ is Axiom A.
\end{theorem}

Thus $X$ is Axiom A. But, since the periodic orbits are dense we obtain that $X$ is Anosov.






\section{Further Remarks}




All the tools used in this paper are available in the context of incompressible flows for manifolds with dimension greater than three. We say that a vector field $X$ in a Riemannian closed manifold $M$ generates an incompressible flow if $div(X)=0$. This implies that for each $t$, the diffeomorphism $X_t$ preserves the Lebesgue measure generated by the Riemannian metric.

We can endow the space of incompressible flows with  the $C^1$ topology and this naturally becomes a Baire space. Also, Kupka-Smale's theorem is available (since the dimension is greater than 3). Franks' lemma and linearizations are also available, see \cite{BDP} and \cite{ArbietoMatheus} as well as Gan-Wen's result (the former follows with the same proof and the linearization given by the pasting lemma, see also \cite{BessaRocha}, and the latter follows from \cite{ArbietoCatalan} and also \cite{Ferreira}). Also, Pugh's general density theorem holds in this setting, see \cite{PR}, as well the connecting lemma, see \cite{WX}. Thus, many of the results obtained here are also valid in this context we the appropriated modifications. For example,

\begin{theorem}\label{globaldiv}
If $X$ is an incompressible vector field which has the weak specification property robustly then $X$ is a topologically mixing Anosov flow.
\end{theorem}

\ack{The authors want to thank Prof. C. Morales for fruitful conversations.}

\end{document}